\documentclass{amsart}

\usepackage{graphicx,algorithm,algpseudocode,xcolor,stmaryrd,amsaddr,tikz,hyperref,amsmath,amssymb}

\algnewcommand\algorithmicinput{\textbf{Input:}}
\algnewcommand\Input{\item[\algorithmicinput]}
\algnewcommand\algorithmicoutput{\textbf{Output:}}
\algnewcommand\Output{\item[\algorithmicoutput]}
\algnewcommand{\LineIf}[2]{\State \algorithmicif\, #1 \,\algorithmicthen\, #2 \,\algorithmicend\ \algorithmicif}
\algnewcommand{\LineForAll}[2]{\State \algorithmicforall\, #1 \,\algorithmicdo\, #2 \,\algorithmicend\ \algorithmicfor}
\algnewcommand{\Accept}{\textbf{accept}}
\algnewcommand{\Reject}{\textbf{reject}}

\newcommand{\dom}{\mathrm{dom}}

\newcommand{\LOGSPACE}{\ensuremath\mathsf{L}}

\newcommand{\gH}{ \mathrel{\mathcal{H}}}

\newtheorem{theorem}{Theorem}[section]
\newtheorem{lemma}[theorem]{Lemma}
\newtheorem{corollary}[theorem]{Corollary}
\newtheorem{proposition}[theorem]{Proposition}

\newtheorem{example}[theorem]{Example}

\newtheorem{problem}[theorem]{Problem}

\numberwithin{equation}{section}
\begin{document}

\title{Answering Five Open Problems Involving Semigroup Conjugacy}
\date{\today}
\author{Trevor Jack}
\address{Illinois Wesleyan University \\
Department of Mathematics \\
1312 Park St.\\
Bloomington, IL 61701}
\email{trevjack@gmail.edu}

\thanks{This work was partially supported by the Funda\c{c}\~{a}o para a Ci\^{e}ncia e a Tecnologia (Portuguese Foundation for Science and Technology) through the projects UIDB/00297/2020 and UIDP/00297/2020 (Centro de Matemática e Aplica\c{c}\~{o}es) and PTDC/MAT-PUR/31174/2017.}

\begin{abstract}
A semigroup conjugacy is an equivalence relation that equals group conjugacy when the semigroup is a group. In this note, we answer five open problems related to semigroup conjugacy. (Problem One) We say a conjugacy $\sim$ is partition-covering if for every set $X$ and every partition of the set, there exists a semigroup with universe $X$ such that the partition gives the $\sim$-conjugacy classes of the semigroup. We prove that six well-studied conjugacy relations -- $\sim_o$, $\sim_c$, $\sim_n$, $\sim_p$, $\sim_{p^*}$, and $\sim_{tr}$ -- are all partition-covering. (Problem Two) For two semigroup elements $a,b \in S$, we say $a \sim_p b$ if there exists $u,v \in S$ such that $a=uv$ and $b=vu$. We give an example of a semigroup that is embeddable in a group for which $\sim_p$ is not transitive. (Problem Three) We construct an infinite chain of first-order definable semigroup conjugacies. (Problem Four) We construct a semigroup for which $\sim_o$ is a congruence and $S / \sim_o$ is not cancellative. (Problem Five) We construct a semigroup for which $\sim_p$ is not transitive while, for each of the semigroup's variants, $\sim_p$ is transitive.
\end{abstract}
\maketitle
\section{Introduction}
Conjugacy classes are important tools for analyzing group structure and significant work has been done to define an equivalence relation for semigroups that equals group conjugacy when the semigroup is a group. It turns out there are manys ways to define such an equivalence relation. In \cite{AK:FN}, Ara\'{u}jo, Kinyon, Konieczny, and Malheiro analyze four ways of defining conjugacy for semigroups. The focus for this note will be answering the following five open problems from their paper.

We say a conjugacy $\sim$ is \emph{partition-covering for a set $X$} if for every partition of the set, there exists a semigroup with universe $X$ such that the partition gives the $\sim$-conjugacy classes of the semigroup. If a conjugacy is covering for every set, we say the conjugacy is \emph{partition-covering}. Conjugacy in groups is certainly not partition-covering for many reasons. For example, there must always be a class with a single element, the identity element. Also, for finite groups, the Orbit-Stabilizer Theorem ensures that the size of each conjugacy class must divide the size of the group. To investigate whether some analogous structure exists for semigroups, Problem 6.10 from \cite{AK:FN} asks whether the conjugacies studied therein are partition-covering. Our main result is that each of the conjugacies are partition-covering. We then discuss several other notions of conjugacy that are not partition-covering.

Problems 6.4 and 6.22 from \cite{AK:FN} are concerned with the following relation. For a semigroup $S$ and any two elements $a,b \in S$, $a \sim_p b$ iff there exists $u,v \in S$ such that $a=uv$ and $b=vu$. This relation is not always transitive, but it is equal to group conjugacy when the semigroup is a group, We know that $\sim_p$ is transitive for free semigroups \cite{LA:SC}. Since free semigroups are embeddable in groups, Problem 6.4 asks whether $\sim_p$ is transitive for any semigroup that can be embedded in a group. We construct a semigroup that can be embedded in a group for which $\sim_p$ is not transitive. Problem 6.22 from \cite{AK:FN} asks whether $\sim_p$ must be transitive for a semigroup if it is transitive for all of the semigroup's variants. We give a counterexample and thus give a negative answer to the problem.

Problem 6.16 from \cite{AK:FN} is concerned with the following relation. For a semigroup $S$ and any two elements $a,b \in S$, $a \sim_o b$ iff there exists $u,b \in S$ such that $au=ub$ and $va=bv$. The problem asks if there is a semigroup for which $\sim_o$ is a congruence and $S / \sim_o$ is not cancellative. We define such an example that has infinite size and we leave as an open question whether a finite example exists.

Conjugacies can be partially ordered by inclusion, which gives rise to a lattice of conjugacies. Problem 6.21 from \cite{AK:FN} asks whether there exists an infinite set of first-order definable conjugacies that forms: (1) a chain ordered by inclusion or (2) an anti-chain ordered by inclusion. We give a partial answer by constructing an infinite chain of first-order definable conjugacies. We also extend the definition of $\sim_p$ to define an infinite chain of first-order definable relations, each of which is equal to group conjugacy when the semigroup is a group. From this, we define a new class of $\LOGSPACE$-complete problems in which we are given two elements and we are asked whether they are $\sim_p$ related.

\section{Preliminaries} \label{NotationSection}
Define the semigroup relation $\sim_g$ to be reflexive and, for any distinct semigroup elements $a,b \in S$, $a \sim_g b$ iff there is a subgroup $G \subseteq S$ containing $a$ and $b$ such that, for some $g \in G$, $gag^{-1} = b$. A semigroup relation $\sim$ is a \emph{conjugacy relation} iff, for any semigroup $S$ that is also a group, $\sim \, = \, \sim_g$. We now present several approaches for generalizing group conjugacy to semigroups. Note that, for groups, the condition for $\sim_g$ could equivalently be written as $ag = gb$. In \cite{OT:CM}, Otto defined semigroup conjugacy in a similar fashion:
$$\forall a,b \in S \, (a \sim_o b \Leftrightarrow ag = gb \text{ and } bh = ha \text{ for some }g,h \in S).$$

This is an equivalence relation, but it suffers the unfortunate drawback of being the universal equivalence relation for any semigroup $S$ that has a \emph{zero element}, $0$, satisfying $0s = s0 = 0$ for every $s \in S$. Ara\'{u}jo, Konieczny, and Malheiro avoided this by restricting the set from which $g$ and $h$ can be chosen. For a semigroup $S$ with zero element $0$ and any nonzero $a \in S \setminus \{0\}$, they defined $\mathbb{P}(a) := \{g \in S: (ma)g \neq 0 \text{ for any } ma \in S^1a \setminus \{0\}\}$. If $S$ has no zero element, then $\mathbb{P}(a) := S$. Finally, $\mathbb{P}^1(a) := \mathbb{P}(a) \cup \{1\}$ with $1$ being the identity element of $S^1$. They then defined $\sim_c$ in a similar fashion as $\sim_o$ \cite{AK:CS}:
$$a \sim_c b \Leftrightarrow ag = gb \text{ and } bh = ha \text{ for some }g \in \mathbb{P}^1(a) \text{ and some } h \in \mathbb{P}^1(b).$$

Note that $\sim_c \, = \, \sim_o$ if the semigroup does not have a zero element.

Alternatively, note that group elements $a,b \in G$ are conjugate iff there exist $g,h \in G$ such that $a=gh$ and $b=hg$. Accordingly, Lallement defined the following notion of semigroup conjugacy \cite{LA:SC}:
$$\forall a,b \in S \, (a \sim_p b \Leftrightarrow a = uv \text{ and } b=vu \text{ for some } u,v \in S).$$

In general, this relation is not transitive, so we denote the transitive closure of this relation as $\sim_{p^*}$.

A semigroup is an \emph{epigroup} if each of its elements has a power that is in a subgroup. Let $S$ be an epigroup and $s \in S$. Define $s^\omega$ to be the idempotent power of $s$. We can define conjugacy for $S$ as follows, motivated by the representation theory of semigroups \cite{IR:RT}. We say $a \sim_{tr} b$ iff there exists $g,h \in S^1$ such that: $ghg=g$, $hgh=h$, $hg = a^\omega$, $gh = b^\omega$, and $ga^{\omega+1}h = b^{\omega+1}$.

In 2018, Janusz Konieczny proposed the following definition \cite{JK:ND}:
$$a \sim_n b \Leftrightarrow ag=gb, bh=ha, hag=b,\text{ and }gbh=a\text{ for some }g,h\in S^1.$$

For conjugacies $\sim_1$ and $\sim_2$, we say $\sim_1 \, \subseteq \, \sim_2$ if for every semigroup $S$ for which the conjugacies are defined and every pair of elements $a,b \in S$, $a \sim_1 b$ implies $a \sim_2 b$. The inclusion is strict, $\sim_1 \, \subsetneq \, \sim_2$, if there exists a semigroup $S$ for which the conjugacies are defined and there exists elements $a,b \in S$ such that $a \not \sim_1 b$ and $a \sim_2 b$. The following is known:
\begin{equation} \label{eq:lattice}
\sim_n \, \subsetneq \, \sim_p \, \subsetneq \, \sim_{p^*} \, \subsetneq \sim_{tr} \, \subsetneq \, \sim_o \text{ and } \sim_n \, \subsetneq \, \sim_c \, \subsetneq \, \sim_o \text{ \cite{AK:FN,AK:CI}}.
\end{equation}

\section{Partition-Covering Problem} \label{sec:open}
Problem 6.10 from \cite{AK:FN} asks whether $\sim_o$, $\sim_p$, and $\sim_{tr}$ are partition-covering. Brute force calculation using the GAP package \emph{Smallsemi} \cite{DM:GAP} proves that $\sim_o$ and $\sim_p$ are partition-covering for any $|X| \leq 6$ \cite{AK:FN}. We will now prove that $\sim_o$, $\sim_c$, $\sim_n$, $\sim_p$, $\sim_{p^*}$, and $\sim_{tr}$ are partition-covering for any set. We start by defining the following relations. For any $a,b \in S$:
\[a \sim_\ell b \Leftrightarrow (ab=a \text{ and } ba=b)\]
\[a \sim_r b \Leftrightarrow (ab=b \text{ and } ba=a).\]
In other words, $a \sim_\ell b$ ($a \sim_r b$) iff $\{a,b\}$ is a left-(right-)zero subsemigroup of $S$.
\begin{lemma} \label{lem:pcband}
For any band $S$, $\sim_\ell$ and $\sim_r$ are both equivalence relations and both are contained in $\sim_n$.
\end{lemma}
\begin{proof}
These relations are clearly symmetric. Let $a,b,c \in S$ such that $a \sim_\ell b$ and $b \sim_\ell c$, then $ac = abc = ab = a$ and $ca = cba = cb = c$, proving that $\sim_\ell$ is transitive. The argument for $\sim_r$ is similar. Also, $aa=a$ implies $a \sim_\ell a$ and $a \sim_r a$, so these relations are equivalence relations for bands. Finally, for bands, $a \sim_\ell b$ ($a \sim_r b$) implies $a \sim_n b$ by picking $g=a$ ($g=b$) and $h=b$ ($h=a$) in Konieczny's definition for $\sim_n$ as stated above. 
\end{proof}
A \emph{normal band} is a semigroup $S$ such that, for all elements $a,b,c \in S$: (1) $aa = a$ and (2) $abca = acba$.
\begin{lemma} \label{lem:partcov}
A conjugacy relation $\sim$ is partition-covering if $\sim_\ell \, \subseteq \, \sim \, \subseteq \, \sim_{tr}$ or $\sim_r \, \subseteq \, \sim \, \subseteq \, \sim_{tr}$ for normal bands.
\end{lemma}
\begin{proof}
Suppose $\sim_\ell \, \subseteq \, \sim \, \subseteq \, \sim_{tr}$ for all normal bands. Let $X$ be any set and $P(X) = \{X_i\}_{i \in \mathcal{I}}$ be any partition of $X$. We will use Von Neumann's definition of an ordinal as a well-ordered set of all smaller ordinals. Using the axiom of choice, every set $X_i$ can be well-ordered and is then order isomorphic to some ordinal $\alpha_i$. So, we can identify $X_i$ with the set $\{(i,j):j \in \alpha_i\}$ and we will refer to the elements of $X_i$ by these pairs. Let $(P(X),<)$ be a total order such that $X_i < X_j$ whenever $\alpha_i \subsetneq \alpha_j$. We will not be concerned about how sets of equal cardinality are ordered. For every $i,j \in \mathcal{I}$, let $\max_P(i,j)=j$ if $X_i < X_j$ and $\max_P(i,j)=i$ otherwise. To define a semigroup whose $\sim$-conjugacy classes are $P(X)$, we define the following multiplication on $X$: $(a,b)(c,d) = (\max_P(a,c),b)$.

To verify that $(X,\cdot)$ is a semigroup, we need to check that, for every $(a,b),(c,d) \in X$, $(a,b)(c,d) \in X$ and that the defined multiplication is associative. If $\max_P(a,c) = a$, then $(a,b)(c,d) = (a,b) \in X_a$. Notably, each $(X_i,\cdot)$ is a left-zero subsemigroup. If $\max_P(a,c) = c$, then $\alpha_a \subseteq \alpha_c$, ensuring that $(c,b) \in X_c$. The multiplication is associative since $(X,\cdot)$ is a semilattice of left-zero subsemigroups. For this same reason, the semigroup is also a normal band.

We now want to show that, for any $(a,b),(c,d) \in X$, the following are equivalent:
\begin{enumerate}
\item $a=c$;
\item $(a,b) \sim_\ell (c,d)$; and
\item $(a,b) \sim_{tr} (c,d)$.
\end{enumerate}

(1) $\Rightarrow$ (2): By definition, $(a,b)(c,d) = (a,b)$ and $(c,d)(a,b)=(c,d)$.

(2) $\Rightarrow$ (3): $\sim_\ell \subseteq \sim_n$ by Lemma~\ref{lem:pcband} and $\sim_n \subseteq \sim_{tr}$ by (\ref{eq:lattice}).

(3) $\Rightarrow$ (1): Because $(a,b) \sim_{tr} (c,d)$, there exists $(g_1,g_2),(h_1,h_2) \in X$ such that $(g_1,g_2)(h_1,h_2) = (a,b)$ and $(h_1,h_2)(g_1,g_2) = (c,d)$. Because $(P(X),<)$ is a total order, then $\max_P(g_1,h_1) = \max_P(h_1,g_1)$ and thus $a=c$.\\

Consequently, for the constructed semigroup, $\sim_\ell \, = \, \sim \, = \, \sim_{tr}$ and the $\sim$-conjugacy classes are given by $P(X)$. By similar argument, if $\sim_r \, \subseteq \, \sim \, \subseteq \, \sim_{tr}$, then we can show $\sim$ is partition-covering by defining $(a,b)(c,d) = (\max_P(a,c),d)$.
\end{proof}
\begin{theorem} \label{th:prob10}
$\sim_o$, $\sim_c$, $\sim_n$, $\sim_p$, $\sim_{p^*}$, and $\sim_{tr}$ are partition-covering.
\end{theorem}
\begin{proof}
Since $\sim_\ell \, \subseteq \, \sim_n$, then $\sim_n$, $\sim_p$, $\sim_{p^*}$, and $\sim_{tr}$ are all partition-covering by Lemma~\ref{lem:partcov}. We need a different construction for $\sim_o$ and $\sim_c$.

As with the proof of Lemma~\ref{lem:partcov}, let $X$ be any set, $P(X) = \{X_i\}_{i \in \mathcal{I}}$ be any partition of $X$, and $\{(i,j):j \in \alpha_i\}$ be the elements of $X_i$. For any cardinal, there exists a commutative group whose number of elements equals that cardinal. Let $(\mathcal{I},\cdot)$ be one such commutative group with identity element $1_I$. To define a semigroup whose $o$-conjugacy classes are $\{X_i\}_{i \in \mathcal{I}}$, we define the following binary operation on $X$: $(a,b)(c,d) = (a \cdot c, 0)$. For any $(a,b),(c,d) \in X$, $(a \cdot c,0) \in X$, since $0 \in \alpha_i$ for every $i \in \mathcal{I}$. Also, for any $(a,b),(c,d),(e,f) \in X$, $(a,b)((c,d)(e,f)) = (a \cdot (c \cdot e),0) = ((a \cdot c) \cdot e,0) = ((a,b)(c,d))(e,f))$, since $(\mathcal{I},\cdot)$ is a group.

We prove that the partition equals the $\sim_o$-conjugacy classes by proving $(a,b) \sim_o (c,d)$ iff $a=c$. Suppose there exists $(e,f) \in X$ such that $(a,b)(e,f) = (e,f)(c,d)$. Then $(a \cdot e,0)=(e \cdot c,0)$. Because $(\mathcal{I},\cdot)$ is a commutative group, $a=c$. Conversely, for any $(a,b),(a,d) \in X_a$, $(a,0) = (a,b)(1_I,0) = (1_I,0)(a,d) = (1_I,0)(a,b) = (a,d)(1_I,0)$.

If $P(X)$ is not a singleton, then the semigroup does not have a zero element and, thus, $P(X)$ is also covered by $\sim_c$. For $P(X) = \{X\}$, we use a third construction, defining $ab=a$ for all $a,b \in X$. Then $a \sim_c b$ for every $a,b \in X$ since $a(a) = (a)b$ and $(b)a = b(b)$. That is, $\sim_c$ is partition-covering.
\end{proof}

\section{Non-Partition Covering Conjugacies}
A natural question to ask at this stage is whether there are any semigroup conjugacies that are \emph{not} partition-covering. Indeed there are. Consider the first relation we defined, $\sim_g$. The smallest semigroup that has distinct elements $a$ and $b$ such that $a \sim_g b$ is the symmetric group $S_3$. But there is another well-studied conjugacy that is not partition-covering.

Let $U(S)$ be the group of units for a semigroup $S$. Then for any $a,b \in S$, we define $a \sim_u b$ iff there exists $u \in U(S)$ such that $u^{-1}au=b$ and $ubu^{-1} = a$ \cite{KM:OC}. Kudryavtseva and Mazorchuk note that, for transformation semigroups, this notion of conjugacy has a nice connection to graph isomorphisms. Let $T_n$ be the full transformation semigroup over the points $\{1,\dots,n\}$. For $a \in T_n$, define the graph $\Gamma(a)$ to have vertices $\{1,\dots,n\}$ and edges $\{(i,j):ia=j\}$. Then for any $a,b \in T_n$, $a \sim_u b$ iff $\Gamma(a) \cong \Gamma(b)$ \cite[Prop 6]{KM:OC}.

Note that $\sim_u$ will only relate distinct elements if $U(S)$ contains more than just the identity element. Since the identity element will always be in its own conjugacy class, $\sim_u$ cannot be partition-covering for any set with more than one element. The smallest group that has distinct conjugate elements is the symmetric group $S_3$. There is a smaller example for $\sim_u$. The following is the multiplication table for \emph{SmallSemigroup}(4, 96) of \cite{DM:GAP}:

\begin{center}
\begin{tabular}{ c| c | c | c | c | c |}
$\cdot$ & 0 & 1 & 2 & 3\\
\hline
0 & 0 & 1 & 2 & 3 \\ 
\hline
1 & 1 & 0 & 2 & 3 \\ 
\hline
2 & 2 & 3 & 2 & 3 \\ 
\hline
3 & 3 & 2 & 2 & 3 \\ 
\hline
\end{tabular}
\end{center}

$2 \sim_u 3$ since $1 \cdot 2 \cdot 1=3$ and $1 \cdot 3 \cdot 1 = 2$.

Certainly, any conjugacy $\sim$ satisfying $\sim \, \subsetneq \, \sim_u$ will also fail to be partition-covering. So another natural question is whether $\sim_u$ is the smallest conjugacy that does not equal $\sim_g$. We now prove that it is not.

Note that both constructions in Lemma~\ref{lem:partcov} and Theorem~\ref{th:prob10} rely on idempotents being conjugate to each other. We can prevent idempotents from being conjugate by using Green's $\gH$ relation. Recall that for any semigroup $S$ and any $a,b \in S$: $a \gH b$ iff ($aS^1 = bS^1$ and $S^1a=S^1b$). Let $H_a := \{ b \in S: a \gH b\}$. Then $\sim_o \cap \gH$ is an equivalence relation. If the semigroup is a group, then all of its elements are $\gH$-related meaning that $\sim_o \cap \gH \, = \, \sim_g$. So, $\sim_o \cap \gH$ is a conjugacy. Similarly, we can consider intersecting $\gH$ with $\sim_{tr}$, $\sim_{p^*}$, $\sim_n$, and $\sim_u$ to obtain four more conjugacies. From \ref{eq:lattice}, we know:
$$\sim_n \cap \gH  \, \subseteq \, \sim_{p^*} \cap \gH  \, \subseteq \, \sim_{tr} \cap \gH  \, \subseteq \, \sim_o \cap \gH$$

Certainly, $\sim_u \cap \gH \, \subseteq \, \sim_n \cap \gH$, since the $u$ and $u^{-1}$ in the definition of $\sim_u$ can serve as the $g$ and $h$ in the definition for $\sim_n$. The rest of this section will prove, by examples, the following structure theorem.
\begin{theorem} \label{thm:structure}
$$\sim_g \, \subsetneq \, \sim_u \cap \gH \, \subsetneq \, \sim_n \cap \gH  \, \subsetneq \, \sim_{p^*} \cap \gH  \, \subsetneq \, \sim_{tr} \cap \gH  \, \subsetneq \, \sim_o \cap \gH,$$
$$\sim_u \cap \gH \, \subsetneq \, \sim_u \, \not\subseteq \, \sim_o \cap \gH, \,\sim_n \cap \gH \, \not \subseteq \, \sim_u, \, \sim_{p^*} \cap \gH \, \not \subseteq \, \sim_n, \, \text{ and } \, \sim_o \cap \gH \, \not \subseteq \, \sim_{tr}.$$
\end{theorem}

The following diagram illustrates Theorem~\ref{thm:structure}. Each solid line represents strict inclusion. Each dotted line represents non-inclusion. We leave as an open problem whether or not $\sim_{tr}\cap \gH \, \subseteq \, \sim_{p^*}$. 

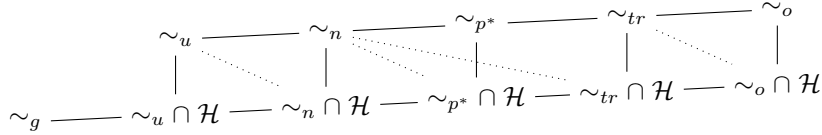
\begin{tikzpicture}
    \node (1) at (-2, 0) {$\sim_g$};
    \node (2) at (0, 1.1) {$\sim_u$};
    \node (3) at (2, 1.2) {$\sim_n$};
    \node (4) at (4, 1.3) {$\sim_{p^*}$};
    \node (5) at (6, 1.4) {$\sim_{tr}$};
    \node (6) at (8, 1.5) {$\sim_o$};
    \node (8) at (0, .1) {$\sim_u \cap \gH$};
    \node (9) at (2, .2) {$\sim_n \cap \gH$};
    \node (10) at (4, .3) {$\sim_{p^*} \cap \gH$};
    \node (11) at (6, .4) {$\sim_{tr} \cap \gH$};
    \node (12) at (8, .5) {$\sim_o \cap \gH$};

    \draw[-] (1) to (8);
    \draw[-] (2) to (3);
    \draw[-] (3) to (4);
    \draw[-] (4) to (5);
    \draw[-] (5) to (6);
    \draw[-] (8) to (9);
    \draw[-] (9) to (10);
    \draw[-] (10) to (11);
    \draw[-] (11) to (12);
    \draw[-] (8) to (2);
    \draw[-] (9) to (3);
    \draw[-] (10) to (4);
    \draw[-] (11) to (5);
    \draw[-] (12) to (6);
    \draw[dotted] (2) to (9);
    \draw[dotted] (3) to (10);
    \draw[dotted] (3) to (11);
    \draw[dotted] (5) to (12);
\end{tikzpicture}

Consider \emph{SmallSemigroup}(5, 255) of \cite{DM:GAP}, whose multiplication table is:

\begin{center}
\begin{tabular}{ c| c | c | c | c | c |}
$\cdot$ & 0 & 1 & 2 & 3 & 4\\
\hline
0 & 0 & 0 & 0 & 0 & 0 \\ 
\hline
1 & 0 & 0 & 0 & 1 & 2 \\ 
\hline
2 & 0 & 0 & 0 & 2 & 1 \\ 
\hline
3 & 0 & 1 & 2 & 3 & 4 \\ 
\hline
4 & 0 & 2 & 1 & 4 & 3 \\ 
\hline
\end{tabular}
\end{center}

The equivalence classes of $\sim_o \cap \gH$, $\sim_{tr} \cap \gH$, and $\sim_{p^*} \cap \gH$ are $\{\{0\},\{1,2\},\{3,4\}\}$, $\{\{0\},\{1,2\},\{3\},\{4\}\}$, and $\{\{0\},\{1\},\{2\},\{3\},\{4\}\}$, respectively. Using GAP \cite{DM:GAP}, we verified that $\sim_n \cap \gH \, = \, \sim_{p^*} \cap \gH$ for all semigroups with 6 or fewer elements. Example~\ref{ex:nNEQp} demonstrates that $\sim_n \cap \gH \, \subsetneq \, \sim_{p^*} \cap \gH$. To define the semigroup, we use notation adopted by GAP for transformations. If $s \in T_4$ is a transformation with domain $\{1,2,3,4\}$ such that $1s = a$, $2s = b$, $3s = c$, and $4s = d$ with $\{a,b,c,d\} \subseteq \{1,2,3,4\}$, then we denote $s$ as $[a,b,c,d]$.

\begin{example}\label{ex:nNEQp}
\begin{align*}S:= \{&[ 3, 2, 1, 1 ], [ 1, 2, 3, 3 ], [ 3, 2, 4, 3 ], [ 3, 2, 4, 4 ], [ 4, 2, 3, 3 ],\\&[ 1, 2, 1, 1 ], [ 3, 2, 3, 3 ], [ 4, 2, 3, 4 ], [ 4, 2, 4, 4 ]\}\end{align*}
\end{example}

Note that:
\begin{enumerate}
\item $[3,2,4,4] \gH [4,2,3,3]$, since $[3,2,4,4][3,2,4,3] = [4,2,3,3]$, $[4,2,3,3][3,2,4,3] = [3,2,4,4]$, $[3,2,1,1][3,2,4,4] = [4,2,3,3]$, and $[3,2,1,1][4,2,3,3] = [3,2,4,4]$.
\item $[3,2,4,4] = [3,2,1,1][4,2,3,3]$ and $[1,2,1,1] = [4,2,3,3][3,2,1,1]$, 
\item $[4,2,3,3] = [3,2,1,1][3,2,4,3]$ and $[1,2,1,1] = [3,2,4,3][3,2,1,1]$, and
\item $[3,2,4,4] \not \sim_n [4,2,3,3]$.
\end{enumerate}

The following example proves that $\sim_u \cap \gH \, \subsetneq \sim_n \cap \gH$.

\begin{example}
$S:=\langle [4,2,3,4], [2,2,3,4], [3,4,2,3], [4,4,3,2]\rangle$
\end{example}
$S$ has 12 elements, the non-generator elements being:
$$\{[2,4,3,2],[4,4,2,3],[2,3,4,2],[3,2,4,3],[3,3,2,4],[4,3,2,4],[3,3,4,2],[2,2,4,3]\},$$

The nontrivial $(\sim_n \cap \gH)$-conjugacy classes are:
$$\{\{[3,4,2,3],[2,3,4,2]\},\{[4,4,3,2], [3,3,2,4], [2,2,4,3]\},$$
$$\{[2,4,3,2],[3,2,4,3],[4,3,2,4]\},\{[4,4,2,3],[3,3,4,2]\}.$$
Since the semigroup does not have a unit, its $(\sim_u \cap \gH)$-conjugacy classes are all singleton sets.

The following example demonstrates that $\sim_g \, \subsetneq \, \sim_u \cap \gH$.

\begin{example}
$S:=\langle [4,4,2,4],[4,2,4,2],[1,4,3,2],[4,3,4,3]\rangle$
\end{example}
$S$ has 13 elements, the non-generator elements being:
$$\{[4,4,4,4],[2,2,2,2],[2,2,4,2],[3,3,3,3],[2,4,2,4],$$
$$[1,2,3,4],[2,3,2,3],[3,4,3,4],[3,2,3,2]\}.$$
It has one pair of elements that are conjugate with respect to $\sim_u \cap \gH$: $[2,4,2,4]$ and $[4,2,4,2]$. Neither of these elements are subgroup elements, so $S$ has no $\sim_g$-conjugate elements. We leave as an open problem whether $\sim_u \cap \gH \, = \, \sim_g$ for all semigroups of size less than 13.

For \emph{SmallSemigroup(4, 96)} of \cite{DM:GAP}, discussed above, elements $2$ and $3$ are not $\gH$-related, so $\sim_u \, \not \subseteq \, \sim_o \cap \gH$. The following example demonstrates that $\sim_n \cap \gH \, \not \subseteq \, \sim_u$.
\begin{example} \label{ex:hnequ}
\begin{align*}S:= \{&[1,2,1,1], [3,3,4,1 ], [ 1,1,4,3], [1,1,3,4], [3,3,3,3],\\&[1,1,1,1], [4,4,1,3], [4,4,3,1], [3,3,1,4], [4,4,4,4]\}\end{align*}
\end{example}
$S$ has no units, so its $\sim_u$-conjugacy classes are trivial. The $(\sim_n \cap \gH)$-conjugacy classes are $\{\{[3,3,4,1],[4,4,1,3]\},\{[1,1,4,3],[4,4,3,1],[3,3,1,4]\}\}$. For example, $[3,3,4,1]\sim_n[4,4,1,3]$ by picking $g=h=[1,1,4,3]$.

\begin{example}
\begin{align*}&S:= \{[1,1,1,1],[4,3,3,4],[1,3,3,1],[4,3,4,3],[4,4,4,4],\\
&[3,4,4,3],[1,3,1,3],[3,4,3,4],[3,3,3,3],[3,1,1,3],[3,1,3,1]\}\end{align*}
\end{example}

Note that $[3,1,3,1] \gH [1,3,1,3]$ since $[3,1,3,1][3,1,1,3] = [1,3,1,3]$, $[1,3,1,3]$ $[3,1,1,3] = [3,1,3,1]$, $[4,3,4,3][3,1,3,1] = [1,3,1,3]$, and $[4,3,4,3][1,3,1,3] = [3,1,3,1]$. Also, $[3,1,3,1] \sim_{p^*} [1,3,1,3]$ since $[3,1,1,3][1,3,1,3] = [1,1,1,1]$, $[1,3,1,3]$ $[3,1,1,3] = [3,1,3,1]$, $[1,3,3,1][1,3,1,3] = [1,1,1,1]$, and $[1,3,1,3]$ $[1,3,3,1] = [1,3,1,3]$. But, $[3,1,3,1] \not \sim_n [1,3,1,3]$ since the only elements with the correct kernal and image are the elements themselves, which do not satisfy the definition of $\sim_n$. Thus, $\sim_{p^*} \cap \gH \not \subseteq \sim_n$.

\begin{example}
\begin{align*}&S:= \{[ 2, 3, 2, 3 ],[ 2, 1, 1, 1 ],[ 2, 3, 2, 2 ],[ 3, 2, 3, 3 ],[ 3, 2, 3, 2 ],\\
&[ 1, 1, 1, 1 ],[ 3, 2, 2, 2 ],[ 1, 2, 2, 2 ],[ 2, 3, 3, 3 ],[ 2, 2, 2, 2 ],[ 3, 3, 3, 3 ]\}\end{align*}
\end{example}
Note that $[3,2,2,2] \sim_{tr} [2,3,3,3]$ by letting $g = [2,2,2,2]$ and $h=[3,3,3,3]$. Also, $[2,1,1,1][3,2,2,2]=[2,3,3,3]$, $[2,1,1,1][2,3,3,3]=[3,2,2,2]$, $[3,2,2,2]$\\
$[2,3,2,3]=[2,3,3,3]$, and $[2,3,3,3][2,3,2,3] = [3,2,2,2]$, so $[3,2,2,2] \gH [2,3,3,3]$. However, there are no elements $u,v \in S$ such that $uv = [3,2,2,2]$ and $vu = [2,3,3,3]$, so $[3,2,2,2] \not \sim_p [2,3,3,3]$. This is evident from the following multiplication table, for which elements are enumerated by the order they appear in the set defined above. In particular, $[3,2,2,2]$ is element $7$ and $[2,3,3,3]$ is element $9$. Consequently, $\sim_{tr} \cap \gH \, \not \subseteq \, \sim_p$. Since $\sim_n \, \subseteq \, \sim_p$, then $\sim_{tr} \cap \gH \, \not \subseteq \, \sim_n$.

\begin{center}
\begin{tabular}{ c | c | c | c | c | c | c | c | c | c | c | c |}
$\cdot$ & 1 & 2 & 3 & 4 & 5 & 6 & 7 & 8 & 9 & 10 & 11\\ \hline
1 & 5 & 6 & 5 & 1 & 1 & 6 & 10 & 10 & 11 & 10 & 11\\ \hline
2 & 7 & 8 & 7 & 9 & 9 & 6 & 9 & 2 & 7 & 10 & 11\\ \hline
3 & 4 & 6 & 4 & 3 & 3 & 6 & 10 & 10 & 11 & 10 & 11\\ \hline
4 & 3 & 6 & 3 & 4 & 4 & 6 & 10 & 10 & 11 & 10 & 11\\ \hline
5 & 1 & 6 & 1 & 5 & 5 & 6 & 10 & 10 & 11 & 10 & 11\\ \hline
6 & 10 & 10 & 10 & 11 & 11 & 6 & 11 & 6 & 10 & 10 & 11\\ \hline
7 & 9 & 6 & 9 & 7 & 7 & 6 & 10 & 10 & 11 & 10 & 11\\ \hline
8 & 9 & 2 & 9 & 7 & 7 & 6 & 7 & 8 & 9 & 10 & 11\\ \hline
9 & 7 & 6 & 7 & 9 & 9 & 6 & 10 & 10 & 11 & 10 & 11\\ \hline
10 & 11 & 6 & 11 & 10 & 10 & 6 & 10 & 10 & 11 & 10 & 11\\ \hline
11 & 10 & 6 & 10 & 11 & 11 & 6 & 10 & 10 & 11 & 10 & 11\\
\end{tabular}
\end{center}

\begin{example}
\emph{SmallSemigroup}(3,11) of \cite{DM:GAP} has the following multiplication table:
\begin{center}
\begin{tabular}{ c | c | c | c |}
$\cdot$ & 0 & 1 & 2 \\ \hline
0 & 0 & 1 & 2 \\  \hline
1 & 1 & 0 & 2 \\  \hline
2 & 2 & 2 & 2 \\ 
\hline
\end{tabular}
\end{center}
\end{example}
Note that $\{0,1\}$ is a subgroup, so $0 \gH 1$. Also, the semigroup has a zero element, so $\sim_o$ is the universal relation. We prove that $1 \not \sim_{tr} 0$ by considering every possible value for $g$ from the definition for $\sim_{tr}$ given above. Note that, for every value, $ghg = g$ and $hgh=h$ together imply that $h=g$. If $g = 2$, then $g1h=2 \neq 0$. In every other case, $g1h=1 \neq 0$. So, $0 \not \sim_{tr} 1$ and thus $\sim_o \cap \gH \, \not \subseteq \, \sim_{tr}$.

\section{\texorpdfstring{$\sim_p$-Transitivity Problems}{~p-Transitivity Problems}}
We now answer Problem 6.2 from \cite{AK:FN} by defining a semigroup for which $\sim_p$ is not transitive and the semigroup is embeddable in a group. We first need some notation. Let $S$ be a finitely presented monoid, $S:=\langle X\, | \, R\rangle$. Then its \emph{left graph} has vertices $X$ and an undirected edge $(a_1,b_1)$ for each relation $a_1 \cdots a_i = b_1 \cdots b_j \in R$. For example, the left graph of $S:=\langle a,b,c \, | \, ab = bc, aa = bb, cb = ca \rangle$ would have three edges: two edges between vertices $a$ and $b$ and a loop from $c$ to itself. The \emph{right graph} has the same vertices and an undirected edge $(a_i,b_j)$ for each relation $a_1 \cdots a_i = b_1 \cdots b_j \in R$. The right graph of $S$ has an edge between $b$ and $c$ and two edges between $b$ and $a$. Adyan proved the following theorem in 1960. \cite{SA:OE,BF:WP}

\begin{theorem} \label{thm:embedding}
A finitely presented monoid is embeddable in a group if neither its left nor right graphs have cycles.
\end{theorem}

\begin{corollary} \label{cor:ptran}
There exists a semigroup that is embeddable in a group for which $\sim_p$ is not transitive.
\end{corollary}
\begin{proof}
Let $S:=\langle a,b \, | \, aab=bba\rangle$. Both the left and right graphs of $S$ contain one edge between distinct points, so neither graph contains cycles. Thus, by Theorem~\ref{thm:embedding}, $S$ can be embedded in a group. Note also that $aba \sim_p aab$, $bba \sim_p bab$, and $aba \not \sim_p bab$. Since $aab=bba$, then $\sim_p$ is not transitive for $S$.
\end{proof}

Problem 6.22 from \cite{AK:FN} asks if $\sim_p$ must be transitive for a semigroup if it is transitive for each of the semigroup's variants. A \emph{variant} of semigroup $(S,\cdot)$ at $a \in S$ is defined as $(S,\circ)$ where for every $x,y \in S$, $x \circ y := x \cdot a \cdot y$. We now present an example of a semigroup for which $\sim_p$ is not transitive even though, for each of the semigroup's variants, $\sim_p$ is transitive.

\begin{example}
$\emph{SmallSemigroup}(8, 1843112331)$ from \cite{DM:GAP} with multiplication table:
\begin{center}
\begin{tabular}{ c | c | c | c | c | c | c | c | c |}
$\cdot$ & 0 & 1 & 2 & 3 & 4 & 5 & 6 & 7\\
\hline
0 & 0 & 0 & 0 & 0 & 0 & 0 & 0 & 0\\
\hline
1 & 0 & 0 & 0 & 0 & 0 & 0 & 0 & 0\\
\hline
2 & 0 & 0 & 0 & 0 & 0 & 0 & 0 & 0\\
\hline
3 & 0 & 0 & 0 & 0 & 0 & 0 & 0 & 0\\
\hline
4 & 0 & 0 & 0 & 0 & 0 & 0 & 0 & 0\\
\hline
5 & 0 & 0 & 0 & 0 & 0 & 0 & 1 & 1\\
\hline
6 & 0 & 0 & 0 & 0 & 0 & 2 & 0 & 3\\
\hline
7 & 0 & 0 & 0 & 0 & 0 & 4 & 3 & 0\\
\hline
\end{tabular}
\end{center}
\end{example}
This example is generated by $\{5,6,7\}$ with the following properties: (1) $5 \cdot 6 = 5 \cdot 7$, (2) $6 \cdot 5 \not \sim_p 7 \cdot 5$, and (3) any word of length at least 3 equals $0$. (1) and (2) imply $\sim_p$ is not transitive for $S$, but (3) implies $\sim_p$ is transitive for every variant.

\section{Noncancellative Quotient Example}
A relation $\sim$ on a semigroup $S$ is a \emph{congruence} if for every $a,b,c,d \in S$, $a \sim b$ and $c \sim d$ implies that $ac \sim bd$. Problem 6.16 from \cite{AK:FN} asks whether there exists an example of a semigroup $S$ for which $\sim_o$ is a congruence and $S/\sim_o$ is not cancellative. The following is such an example.
\begin{example}
Let $S:=\{a,b \, | \, a^2=1, \, ab = b\} = \{1,b,b^2,\dots,a,ba,b^2a,\dots\}$.
\end{example}
We first claim that the following are the $\sim_o$-conjugacy classes:
$$\{\{1\},\{a\}\} \cup \{ \{b^i,b^ia\}\}_{i \in \mathbb{N}}.$$

For each $i \in \mathbb{N}$, $b^i \sim_o b^ia$ since $b^i(ba) = b^{i+1}a = (ba)b^ia$ and $(b)b^i = b^{i+1} = b^ia(b)$. For any distinct $i,j \in \{0,1,\dots\}$, we want to show that $b^i \not \sim_o b^j$. Note that, for any $k \in \{0,1,\dots\}$ and $\ell \in \{0,1\}$, $b^i(b^ka^\ell) = b^{i+k}a^\ell \neq b^{j+k} = (b^ka^\ell)b^j$. Also, $b^ka^\ell a = b^ka^{1-\ell} \neq b^ka^\ell = 1b^ka^\ell$, so $a \not \sim_o 1$. To prove $\sim_o$ is a congruence, pick any two pairs of congruent elements $b^ia^k \sim_o b^ia^\ell$ and $b^ja^m \sim_o b^ja^n$. Then $b^ia^kb^ja^m = b^{i+j}a^m \sim_o b^{i+j}a^n = b^ia^\ell b^ja^n$. We conclude by noting that $ab = b = 1b$, which shows that $S / \sim_o$ is not right-cancelative.

\section{Conjugacy Chain Problem}
Figure 1.1 from \cite{AK:FN} depicts a lattice of conjugacies in which order is defined by inclusion. Problem 6.21 from \cite{AK:FN} asks whether there exists an infinite set of first-order definable conjugacies that forms (a) an antichain or (b) a chain. We now construct an infinite chain of first-order definable conjugacies. First, we need the following notation.

The \emph{full partial bijection semigroup} over $[n]:=\{1,\dots,n\}$, denoted $I_n$, is the set of all bijections defined on subsets of $[n]$, together with function composition. Denote the domain and image of $a \in I_n$ as $\dom(a)$ and $[n]a$, respectively. Every element $s \in I_n$ can be characterized by the orbits of its action on $[n]$, each of which is one of two types. A \emph{cycle of length k} is an ordered subset of $[n]$, denoted $(x_1,\dots,x_k)$, such that $x_1 = \min(x_1,\dots,x_k)$, $x_is = x_{i+1}$ for $1 \leq i < k$, and $x_ks = x_1$. Let $\Delta_s$ be the set of all cycles and $\Delta_s^k$ be the set of all cycles of length $k$. By this convenction, $\Delta_s^1 := \{{x} : xs = x\}$.

A \emph{chain of length k} is an ordered set denoted $[x_1,\dots,x_k]$ such that $x_1 \not \in [n]s$, $x_is = x_{i+1}$ for $1 \leq i < k$, and $x_k \not \in \dom(s)$. Let $\Theta_s$ be the set of all chains, $\Theta_s^k$ be the set of all chains of length $k$, and $\Theta_s^{>k}$ be the set of all chains of length greater than $k$. By this convention, $\Theta_s^1 := \{\{x\}:x \not \in \dom(s)\}$. For $s \in I_n$, its \emph{cycle type} and \emph{chain type} are the sequences $\langle |\Delta_s^1|,\dots,|\Delta_s^n|\rangle$ and $\langle |\Theta_s^1|,\dots,|\Theta_s^n|\rangle$, respectively. Note that for any $s \in S$, each $x \in [n]$ lies in either a cycle or a chain of $s$.

For any inverse semigroup elements $a,b \in S$, $a \sim_n b$ iff there exists $g \in S^1$ such that $g^{-1}ag=b$ and $gbg^{-1} = a$ \cite[Thm 2.6]{JK:ND}. For any $a,b \in I_n$, it is also known that: (1) $a \sim_n b$ iff $a$ and $b$ have the same cycle-chain type \cite[Thm 2.6]{JK:ND}\cite[Thm 2.10]{AK:CI}; (2) $\sim_{tr} \, = \, \sim_{p^*}$ \cite[Thm 2]{KM:TA}; and (3) $a \sim_{p^*} b$ iff $a$ and $b$ have the same cycle type \cite[Thm 10]{KM:OC}. Note that $a^\omega$ is the idempotent that fixes the points in the cycles of $a$ and omits the points in the chains of $a$. Thus, in the the full inverse semigroup, $a \sim_{tr} b$ iff $a^{\omega+1} \sim_n b^{\omega+1}$, since both relations are characterized by $a$ and $b$ have the same cycle types.

Motivated by this result, for each $k \in \mathbb{N}$, we define the relation $a \sim_{n[k]} b$ as follows. For any $S \leq I_n$ and any $a,b \in S$, $a \sim_{n[k]} b$ iff $a^{k+1}a^{-k} \sim_n b^{k+1}b^{-k}$. That $\sim_{n[k]}$ is an equivalence relation follows immediately from $\sim_n$ being an equivalence relation. Since $a^{k+1}a^{-k} = a$ for all group elements $a$, then $\sim_{n[k]} \, = \, \sim_n \, = \, \sim_g$ when $S$ is a group. We want to prove that $\{\sim_{n[k]}\}_{k \in \mathbb{N}}$ is an infinite chain.

\begin{theorem}
For any $i,j \in \mathbb{N}$, $i<j$ implies $\sim_{n[i]} \subsetneq \sim_{n[j]}$.
\end{theorem}
\begin{proof}
Suppose $a \sim_{n[i]} b$ and let $g \in S^1$ satisfy $g^{-1} a^{i+1} a^{-i} g = b^{i+1} b^{-i}$ and $a^{i+1} a^{-i} = g b^{i+1} b^{-i} g^{-1}$. By \cite[Prop 1.3]{AK:CI}, $g^{-1} a^{i+1} a^{-i} = b^{i+1}b^{-i} g^{-1}$, $a^{i+1} a^{-i} g = g b^{i+1} b^{-i}$, and $b^{i+1}b^{-i}g^{-1}g = b^{i+1}b^{-i}$. Then because $i<j$:
\begin{align*}
g^{-1} a^{j+1}a^{-j} g &= g^{-1} (a^{i+1}a^{-i})^{j+1}(a^ia^{-i-1})^j g\\
&= (b^{i+1}b^{-i})^{j+1}g^{-1}g(b^ib^{-i-1})^j\\
&= b^{j+1}b^{-j}
\end{align*}
Similarly, $g b^{j+1}b^{-j} g^{-1} = a^{j+1}a^{-j}$. Thus, $a \sim_{n[j]} b$.

Define $a \in I_j$ as follows: $xa = x+1$ for $x \in \{1,\dots,j-1\}$ and $j \not \in \dom(a)$. Thus, $a$ has a single chain of length $j$. Define $b \in I_j$ to have empty domain. Then $xa^i = x+i+1$ for $x \in \{1,\dots,j-i-1\}$ and $a^j = b^i = b^j = b$. Since $a^j$ and $b^j$ have the same cycle-chain type, $a \sim_{n[j]} b$. Since $a^i$ and $b^i$ do not have the same cycle-chain type, $a \not \sim_{n[i]} b$.
\end{proof}

Thus, $\{\sim_{n[k]}\}_{k \in \mathbb{N}}$ is an infinite chain of conjugacies. Consider the following set of problems, indexed by $k \in \mathbb{N}$:

\medskip
{\bf n[k]-Conjugacy}
\begin{itemize}
\item Input: $a,b \in I_n$
\item Problem: Is $a \sim_{n[k]} b$?
\end{itemize}

\begin{theorem}
For each $k \in \mathbb{N}$, n[k]-Conjugacy is $\LOGSPACE$-complete.
\end{theorem}
\begin{proof}
We first claim that, for any $a,b \in I_n$, $a \sim_{n[k]} b$ iff $a$ and $b$ have the same cycle type and $\langle |\Theta_a^k|,\dots,|\Theta_a^n|\rangle = \langle |\Theta_b^k|,\dots,|\Theta_b^n|\rangle$. Let $a' = a^{k+1}a^{-k}$ and $b' = b^{k+1}b^{-k}$. Certainly, $xa = xa'$ for every point $x$ in a cycle of $a$. So, $\Delta_a = \Delta_{a'}$ and $\Delta_{b'} = \Delta_b$ and thus $|\Delta_a| = |\Delta_b|$ iff $|\Delta_{a'}| = |\Delta_{b'}|$. Let $[x_1,\dots,x_i] \in \Theta_a^i$ such that $i\geq k$. Then $[x_1,\dots,x_{i-k}] \in \Theta_{a'}^{i-k}$. That is, $|\Theta_a^i| = |\Theta_{a'}^{i-k}|$ and $|\Theta_{b'}^{i-k}| = |\Theta_b^i|$. So, for $i \geq k$, $|\Theta_a^i|=|\Theta_b^i|$ iff $|\Theta_{a'}^{i-k}| = |\Theta_{b'}^{i-k}|$. Our claim then follows from the fact that $a' \sim_n b'$ iff they have the same cycle type and chain type.\cite[Thm 2.6]{JK:ND}\cite[Thm 2.10]{AK:CI}

Algorithms 1 and 2 from \cite{TJ:CI} describe how to deterministically check that two elements have the same cycle and chain types using at most logarithmic space. We can adapt them to our current problem by simply ignoring chains of length less than $k$. The proof for \cite[Thm 4.6]{TJ:CI} shows that testing conjugacy within groups is $\LOGSPACE$-hard, so testing any notion of semigroup conjugacy is $\LOGSPACE$-hard.
\end{proof}

\section{\texorpdfstring{Extending $\sim_p$}{Extending ~p}}
Our first attempt to answer Problem 6.21 relied on extending the definition of $\sim_p$. Although it did not yield a complete answer, it did yield a new class of $\LOGSPACE$-complete problems.

For a semigroup $S$, define $\sim_{p[1]} := \sim_p$ and $\sim_{p[m]} := \{ (a,b) \in S^2: \exists c \in S((a \sim_{p[m-1]} c) \land (c \sim_{p[m-1]} b))$ for every integer $m>1$. Certainly, for every semigroup and every pair of positive integers $m<n$, $\sim_p \subseteq \, \sim_{p[m]} \, \subseteq \, \sim_{p[n]} \, \subseteq \, \sim_{p^*}$. So, when the semigroup is a group, $\sim_{p[m]} = \sim_g$ for every positive integer $m$. To prove the set is an infinite chain, we want to define a semigroup for which $\sim_{p[m]} \, \subsetneq \, \sim_{p[n]}$ for each pair of positive integers $m < n$. To do this, we will need the following notation.

For any $a,b \in I_n$, let $\phi: \Theta_a^{>{2^{m-1}}} \rightarrow \Theta_b$ and $\psi: \Theta_b^{>2^{m-1}} \rightarrow \Theta_a$ be injective mappings. We call $(\phi,\psi)$ an \emph{inversive m-satisfying pair for $a$ and $b$} if for every $\theta \in \Theta_a^{>2^{m-1}}$ and every $\theta' \in \Theta_b^{>2^{m-1}}$: (1) $|\theta| \leq |\phi(\theta)|+2^{m-1}$; (2) $|\theta'| \leq |\psi(\theta')|+2^{m-1}$; and (3) $\phi \circ \psi$ and $\psi \circ \phi$ are both idempotent. Lemma 4.4 from \cite{TJ:CI} with Theorem 10 from \cite{KM:OC} gives us that, for any $a,b \in I_n$, $a \sim_p b$ iff $a \sim_{p^*} b$ and there exists an inversive 1-satisfying pair for $a$ and $b$. We now extend this result to $\sim_{p[m]}$. 

\begin{lemma} \label{lem:invPair}
For any $a,b \in I_n$, $ a \sim_{p[m]} b$ iff: $a \sim_{p^*} b$ and there exists an inversive m-satisfying pair for a and b.
\end{lemma}
\begin{proof}
We proceed by induction on $m$. Case $m=1$ is given by \cite[Lemma 4.4]{TJ:CI}. Assume that Lemma~\ref{lem:invPair} holds for positive integers $i<m$. If $a \sim_{p[m+1]} b$, then there exists $c \in I_n$ such that $a \sim_{p[m]} c$ and $c \sim_{p[m]} b$. By our induction hypothesis, $a \sim_{p^*} c$ and $c \sim_{p^*} b$, so $a \sim_{p^*} b$. Furthermore, there exists an inversive m-satisfying pairs $(\phi_a,\psi_a)$ for $a$ and $c$ and $(\phi_c,\psi_c)$ for $c$ and $b$. We claim that $(\phi_c \circ \phi_a, \psi_a \circ \psi_c)$, restricted to $\Theta_a^{>2^m}$ and $\Theta_b^{>2^m}$, is an inversive (m+1)-satisfying pair for $a$ and $b$.

Pick any $\theta \in \Theta_a^{>2^m}$ and any $\theta' \in \Theta_b^{2^m}$. Then $|\theta| \leq |\phi_a(\theta)|+2^{m-1} \leq |\phi_c(\phi_a(\theta))|+2^{m-1}+2^{m-1}=|\phi_c(\phi_a(\theta))|+2^m$. Likewise, $|\theta'| \leq |\psi_a(\psi_c(\theta'))|+2^m$. Certainly, $\phi_c \circ \phi_a$ and $\psi_a \circ \psi_c$ are both injective. Since $\phi_a \circ \psi_a$ and $\phi_c \circ \psi_c$ are both injective and idempotent, they must fix every point in their respective domains. Consequently, $(\phi_c \circ \phi_a \circ \psi_a \circ \psi_c)(\theta) = \theta$ for every $\theta \in \dom(\psi_a \circ \psi_c)$, meaning $(\phi_c \circ \phi_a \circ \psi_a \circ \psi_c)$ is injective and idempotent. Likewise, $\psi_a \circ \psi_c \circ \phi_c \circ \phi_a$ is injective and idempotent. Thus, $(\phi_c \circ \phi_a, \psi_a \circ \psi_c)$ is an inversive (m+1)-satisfying pair for $a$ and $b$.

Conversely, suppose $a \sim_{p^*} b$ and let $(\phi,\psi)$ be an inversive (m+1)-satisfying pair for $a$ and $b$. We will construct $c \in I_n$ and inversive m-satisfying pairs $(\phi_a,\psi_a)$ for $a$ and $c$ and $(\phi_c,\psi_c)$ for $c$ and $b$. By our induction hypothesis, this will prove $a \sim_{p[m]} c$ and $c \sim_{p[m]} b$, thus proving $a \sim_{p[m+1]} b$. We will construct $c$ and the pairs in two stages.

For the first stage, define $\Theta_0$ to be the disjoint union of $\dom(\phi)$ and $\dom(\psi)$. For integers $0 \leq i$, while $\Theta_{i+1}$ is nonempty, define the following: $\Theta_{i+1} := \Theta_i \setminus \{\theta_i,\theta_i'\}$, where $\theta_i$ is any chain chosen from $\Theta_i$ and $\theta_i' = \phi(\theta_i)$ if $\theta \in \dom(\phi)$ or $\theta_i' = \psi(\theta)$ otherwise. Let $\Theta_\delta$ be the last nonempty set. Define $\alpha_{-1} = 0$ and, for each $0\leq i \leq \delta$,
\begin{equation*}
\alpha_i := \alpha_{i-1} +  \lfloor .5(|\theta_j|+|\theta_j'|)\rfloor.
\end{equation*}

Define $\Theta_c$ to include the following chains $\{[\alpha_{i-1}+1,\alpha_i]:0 \leq i \leq \delta\}$. For each integer $0 \leq i \leq \delta$, define $(\phi_a,\psi_a)$ and $(\phi_c,\psi_c)$ as follows. If $\theta_i \in \dom(\phi)$: $\phi_a(\theta_i) := [\alpha_{i-1},\alpha_i]$, $\phi_c([\alpha_{i-1},\alpha_i]) := \theta_i'$, $\psi_a([\alpha_{i-1},\alpha_i]) := \theta_i$, and $\psi_c(\theta_i') := [\alpha_{i-1},\alpha_i]$. If $\theta_i \in \dom(\psi)$: $\phi_a(\theta_i') := [\alpha_i,\alpha_{i+1}]$, $\phi_c([\alpha_{i-1},\alpha_i]) := \theta_i$, $\psi_a([\alpha_{i-1},\alpha_i]) := \theta_i'$, and $\psi_c(\theta_i) := [\alpha_{i-1},\alpha_i]$.

By construction, all four functions are injective and the following compositions are all idempotent: $\phi_a \circ \psi_a$, $\psi_a \circ \phi_a$, $\phi_c \circ \psi_c$, and $\psi_c \circ \phi_c$. Because $(\phi,\psi)$ is an inversive (m+1)-satisfying pair for $a$ and $b$, for every integer $0\leq i \leq \delta$, $|\theta_i| \leq |\theta_i'|+2^m$ and $|\theta_i'| \leq |\theta_i|+2^m$. That is, $||\theta_i| - |\theta_i'|| \leq 2^m$. Note that, for every $0 \leq i \leq \delta$, $|\theta_i| - (\alpha_i-\alpha_{i-1}) = \lceil .5(|\theta_i|-|\theta_i'|)\rceil$. Thus for every $\theta \in \dom(\phi_a)$, $|\theta| \leq |\phi_a(\theta)| + 2^{m-1}$ and for every $\theta \in \dom(\psi_a)$, $|\theta| \leq |\psi_a(\theta)| + 2^{m-1}$. By similar argument, $|\theta| \leq |\phi_c(\theta)| + 2^{m-1}$ for every $\theta \in \dom(\phi_c)$ and $|\theta| \leq |\psi_c(\theta)| + 2^{m-1}$ for every $\theta \in \dom(\psi_c)$.

We now finish defining $c$ and the pairs using the remaining chains: 
$$\Theta' := (\Theta_a^{>2^{m-1}} \cup \Theta_b^{>2^{m-1}}) \setminus \bigcup \limits_{0 \leq i < \delta} \{\theta_i,\theta_i'\}$$.

Enumerate $\Theta'$ as $\{\theta_{\delta+1},\dots,\theta_\varepsilon\}$. For each integer $\delta<i\leq \varepsilon$, define
\begin{equation*}
\alpha_i := \alpha_{i-1} \lfloor .5|\theta_j|\rfloor
\end{equation*}
Define $\Theta_c$ to include the chains $\{[\alpha_{i-1}+1,\alpha_i]: \delta < i \leq \varepsilon\}$. Note that $\alpha_\epsilon$ was calculated by adding chains from $\Theta_a^{>2^{n-1}}$ and $\Theta_b^{2^{n-1}}$ and dividing that sum by $2$. Since the lengths of the chains of $a$ and $b$ can, at most, add to be $2n$, then $\alpha_\epsilon \leq n$. For each integer $\delta<i\leq\varepsilon$, define $(\phi_a,\psi_a)$ and $(\phi_c,\psi_c)$ as follows. If $\theta_i \in \dom(\phi)$, $\phi_a(\theta_i) = [\alpha_{i-1},\alpha_i]$. If $\theta_i \in \dom(\psi)$, $\psi_c(\theta_i) = [\alpha_{i-1},\alpha_i]$. Note that each $\theta \in \Theta'$ has length at most $2^m$, so $\alpha_i-\alpha_{i-1} \leq 2^{m-1}$. That means $[\alpha_{i-1},\alpha_i] \not \in \Theta_c^{>2^{m-1}}$ and thus is not in the domain of $\psi_a$ nor $\phi_c$. Furthermore, for every integer $\delta<i\leq \varepsilon$, $|\theta_i| - (\alpha_i-\alpha_{i-1}) = \frac{1}{2}\theta_i \leq 2^{m-1}$, thus completing the proof that $(\phi_a,\psi_a)$ and $(\phi_c,\psi_c)$ are inversive m-satisfying pairs for $a$ and $c$ and for $c$ and $b$.
\end{proof}

Consider the following set of problems, indexed by $m \in \mathbb{N}$:

\medskip
{\bf p[m]-Conjugacy}
\begin{itemize}
\item Input: $a,b \in I_n$
\item Problem: Is $a \sim_{p[m]} b$?
\end{itemize}

Let $\LOGSPACE$ denote the complexity class of problems that can be solved deterministically using at most logarithmic space. The following problem is $\LOGSPACE$-complete: given $a,b \in I_n$, determine whether $a \sim_p b$ \cite[Theorem 4.6]{TJ:CI}. Lemma~\ref{lem:invPair} allows us to extend this result to p[m]-Conjugacy for every $m \in \mathbb{N}$.

\begin{theorem}
For each $m \in \mathbb{N}$, p[m]-Conjugacy is $\LOGSPACE$-complete.
\end{theorem}
\begin{proof}
The argument in Theorem 4.6 proves that every semigroup conjugacy is $\LOGSPACE$-hard, so we need only show that p[m]-Conjugacy can be deterministically solved using at most logarithmic space. By Lemma~\ref{lem:invPair}, we need to check that $a \sim_{p^*} b$ and that there exists an inversive $m$-satisfying pair for $a$ and $b$. Lemma 4.5 from \cite{TJ:CI} proves that $a \sim_{p^*} b$ can be deterministically checked using at most logarithmic space. We claim that Algorithm~\ref{alg:p[m]Conj} deterministically checks whether there exists an inversive $m$-satisfying pair for $a$ and $b$.

\emph{Correctness:} The entries in $C$ will track the number of unpaired chains. A positive value indicates the unpaired chains are from $\Theta_a$. A negative value indicates the unpaired chains are from $\Theta_b$. At iteration $k$, the unpaired chains counted by $C[i]$ will be length $i+k$. For the first iteration, $k=n$, Lines 4-16 do nothing, since all the entries of $C$ are zero. The calculation in Line 17 assumes chains from $\Theta_a^n$ will be paired with chains from $\Theta_b^n$ and $C[1]$ stores the number of unpaired chains. On the next $2^{m-1}$ iterations, Lines 4-5 will still do nothing and Lines 8-16 will update the values for $C$, accounting for chains from $\Theta_a^k$ and $\Theta_b^k$ being paired with chains that are at most $2^{m-1}$ longer.

Line 15 increments the indices, since the next iteration will be considering chains of length one less. So, during iteration $k$, $C[2^{m-1}+1]$ represents the number of chains of length $2^{m-1}+1+k$ that could not have been paired in previous iterations. Line 4 checks if these chains can be paired with chains of length $k$. If not, then no inversive $m$-satisfying pair for $a$ and $b$ exists and Line 5 rejects. If Line 5 does not reject by the time $k=0$, then all remaining unpaired chains, as tracked by the entries in $C$, have lengths less than or equal to $2^{m-1}$. Thus, an inversive $m$-satisfying pair for $a$ and $b$ exists and Line 20 accepts.

\emph{Runspace:} Algorithm 2 from \cite{TJ:CI} proves that Line 4 runs in deterministic logarithmic space. Although Algorithm~\ref{alg:p[m]Conj} needs to store values for $t_a, t_b, C[1], \dots, C[2^{m-1}],$ and $k$, this is a constant number of values with respect to the input variable $n$. So this algorithm uses at most logarithmic space.
\end{proof}

  \begin{algorithm}
  \caption{$\LOGSPACE$ algorithm to check if $a \sim_{p[m]} b$ in $I_n$}
    \label{alg:p[m]Conj}
    \begin{algorithmic}[1]
      \Input{$a,b \in I_n$}
      \Output{Is $a \sim_{p[m]} b$ in $I_n$?}
      \State Initialize: $C = [0,\dots,0]$ with indices $i=1,\dots,2^{m-1}+1$
      \State Initialize: $k:=n$
      \While{$k > 0$}
        \State $t_a := |\Theta_a^k|$ and $t_b := |\Theta_b^k|$
        \If{$t_a + C[2^{m-1}+1] < 0$ or $t_b - C[2^{m-1}+1] < 0$}
          \State \Reject
        \EndIf
        \State $i:=2^{m-1}$
        \While{$i>0$}
          \If{$C[i]>0$}
            \State $C[i]:=\max(0,C[i]-t_b)$ and $t_b:=\max(0,t_b-C[i])$
          \EndIf
          \If{$C[i]<0$}
            \State $C[i]:=\min(0,C[i]+a)$ and $t_a:=\max(0,C[i]+t_a)$
          \EndIf
          \State $C[i+1]:=C[i]$ and $i:=i-1$
        \EndWhile
        \State $C[1]:=t_a - t_b$ and $k:=k-1$
      \EndWhile
      \State \Accept
    \end{algorithmic}
  \end{algorithm}

\begin{proposition} \label{lem:simEq}
For $I_n$, $\sim_{p[m]} = \sim_{p^*}$ iff $m \geq \log_2(n-1)+1$.
\end{proposition}
\begin{proof}
By Lemma~\ref{lem:invPair}, the only way $\sim_{p[m]} \neq \sim_{p^*}$ is if there are $a,b \in I_n$ such that $a \sim_{p^*} b$, but there exists no inversive $m$-satisfying pair for $a$ and $b$. If $\Theta_a^{>2^{m-1}}$ and $\Theta_b^{>2^{m-1}}$ are empty, then vacuously, an inversive $m$-satsifying pair exists. If $m \geq \log_2(n-1)+1$, then $2^{m-1} \geq n-1$ and the only chain that could be in either set would be of length $n$. WLOG, let $\theta \in \Theta_a^{>2^{m-1}}$ have length $n$. Then $a$ has no cycles. If $b$ has a cycle, then $a \not \sim_{p^*} b$. If $b$ has no cycle, then there exists $\theta' \in \Theta_b$ and we can define $\phi(\theta) = \theta'$. Note that $|\theta| = 1 + (n-1) \leq |\phi(\theta)|+2^{m-1}$. If $|\theta'|<n$ we need not define $\psi(\theta')$. Otherwise define $\psi(\theta') = \theta$. Likewise, $|\theta'| \leq |\psi(\theta')|+2^{m-1}$.

Suppose $m < \log_2(n-1)+1$ and define $a,b \in I_n$ as follows. Define $xa = x+1$ for $x \in [n-1]$ and $n \not \in \dom(a)$. Define $\dom(b) = \emptyset$. Neither element has cycles, so $a \sim_{p^*} b$. Because $m<\log_2(n-1)+1$, $2^{m-1}<n-1$ and $[1,\dots,n] \in \Theta_a^{>2^{m-1}}$. Since the longest chain in $\Theta_b$ is of length 1, for any $\phi:\Theta_a^{>2^{m-1}} \rightarrow \Theta_b$, $|[1,\dots,n]| = n = 1 + (n-1) > |\phi([1,\dots,n])| + 2^{m-1}$. Thus, no inversive $m$-satisfying pair exists and, by Lemma~\ref{lem:invPair}, $a \not \sim_{p[m]} b$.
\end{proof}

Then, for any two integers $0 < n < m$, $\sim_{p[m]} \, = \, \sim_{p^*} \, \neq \, \sim_{p[n]}$ for $I_{2^n+1}$. Thus, $\{\sim_{p[m]}:m \in \mathbb{N}\}$ forms an infinite chain of first-order definable relations, each of which equals group conjugacy when the semigroup is a group, though none of which are, in general, transitive.

\section{Open Problems}
\begin{problem}
For each $\sim \, \in \{\sim_o,\sim_{tr},\sim_{p^*},\sim_n,\sim_u\}$, what is the largest $n \in \mathbb{N}$ such that $\sim \cap \gH \, = \, \sim_g$ for all semigroups with at most $n$ elements?
\end{problem}

\begin{problem}
Is there an epigroup for which $\sim_{tr} \cap \gH \not \subseteq \sim_{p^*}$?
\end{problem}

\begin{problem}
Is there a first-order definable semigroup conjugacy other than $\sim_g$ that is strictly contained in $\sim_u \cap \gH$?
\end{problem}

\begin{problem}
Is there a finite semigroup for which $\sim_o$ is a congruence and $S / \sim_o$ is not cancelative?
\end{problem}

\end{document}